\documentclass[a4paper,11pt]{article}
\usepackage[T1]{fontenc}
\usepackage{lmodern}
\usepackage{amsmath,amssymb,amsthm}
\usepackage[left=27mm,right=27mm,top=25mm,bottom=25mm]{geometry}
\usepackage[hidelinks]{hyperref}
\newtheorem{theorem}{Theorem}[section]
\newtheorem{lemma}[theorem]{Lemma}
\newcommand{\Z}{\mathbb Z}
\newcommand{\Q}{\mathbb Q}
\newcommand{\N}{\mathbb N}
\newcommand{\Fp}{\mathbb F_p}
\newcommand{\eps}{\varepsilon}
\newcommand{\vp}{\nu_p}
\newcommand{\Zp}{\Z_{(p)}}
\DeclareMathOperator{\lcm}{lcm}
\title{On the denominators of signed harmonic sums}
\author{Quan-Hui Yang\thanks{The author is supported by the National Natural Science Foundation of China (Grants No. 12371005 and 12671008).}\\[4pt]
{\small School of Mathematical Sciences and}\\
{\small Ministry of Education Key Laboratory for NSLSCS,}\\
{\small Nanjing Normal University, Nanjing 210023, P.R. China}\\
{\small\texttt{yangquanhui01@163.com}}}
\date{}
\begin{document}
\maketitle
\begin{abstract}
Let $\eps=(\eps_i)_{i\ge1}$ be a sequence with $\eps_i\in\{-1,1\}$, and let $v_n(\eps)$ be the positive denominator of $\sum_{i=1}^n\eps_i/i$ in lowest terms. We prove that $|\{1\le n\le x:v_n(\eps)\ne v_{n+1}(\eps)\}|=o(x)$ uniformly in $\eps$. In particular, the set of positive integers $n$ with $v_n(\eps)=v_{n+1}(\eps)$ has asymptotic density one for every sequence of signs. This answers a problem of Wu and Chen, \emph{J. Number Theory} \textbf{200} (2019), 397--406.
\end{abstract}

{\small
\noindent\textbf{Keywords:} Signed harmonic sums, denominators, asymptotic density.
\par\noindent\textbf{2020 Mathematics Subject Classification:} 05A16; 11B83.
\par}

\section{Introduction}
Let $\N$ denote the set of positive integers. For a sequence $\eps=(\eps_i)_{i\ge1}$ with $\eps_i\in\{-1,1\}$, write
\[
 H_n(\eps)=\sum_{i=1}^n\frac{\eps_i}{i}
 =\frac{u_n(\eps)}{v_n(\eps)},\qquad
 (u_n(\eps),v_n(\eps))=1,\quad v_n(\eps)>0.
\]
The choices $\eps_i=1$ and $\eps_i=(-1)^{i-1}$ give the ordinary and alternating harmonic numbers, respectively. Put $\mathbf1=(1,1,\ldots)$.

Eswarathasan and Levine~\cite{EL} and Boyd~\cite{Boyd} studied the divisibility of the numerators of ordinary harmonic numbers. Sanna~\cite{Sanna} and Wu and Chen~\cite{WC17} obtained quantitative estimates for the sets arising from their $p$-adic valuations. Wu and Chen~\cite{WC18} proved that, for every positive integer $m$, the set $\{n\in\N:m\mid v_n(\mathbf1)\}$ has asymptotic density one.

For ordinary harmonic numbers, Shiu~\cite{Shiu} proved that $v_n(\mathbf1)<v_{n+1}(\mathbf1)$, $v_n(\mathbf1)=v_{n+1}(\mathbf1)$ and $v_n(\mathbf1)>v_{n+1}(\mathbf1)$ each hold for infinitely many $n$. Wu and Chen~\cite[Theorems~1.1--1.2]{WC} proved that the equality holds on a set of asymptotic density one for both the ordinary and alternating harmonic numbers. They posed the following problem.

\medskip\noindent
\textbf{Problem}~\cite[Problem~1]{WC}.
\textit{For every sequence $\eps\in\{-1,1\}^{\N}$, does the set of positive integers $n$ with $v_n(\eps)=v_{n+1}(\eps)$ have asymptotic density one?}
\medskip

Wu and Yan~\cite[Theorem~1.1]{WY} gave an affirmative answer when $\eps$ is purely periodic, that is, when $\eps_{i+s}=\eps_i$ for some positive integer $s$ and all $i\ge1$. For relatively prime positive integers $a$ and $b$, Chen and Wu~\cite[Corollary~2]{CW21} proved the corresponding density-one result for the denominators of $\sum_{j=0}^{n-1}1/(aj+b)$ and $\sum_{j=0}^{n}1/(aj+b)$.

In this paper, we give an affirmative answer for an arbitrary sequence of signs. We prove the following uniform result.

\begin{theorem}\label{thm:main}
We have
\[
 \lim_{x\to\infty}\sup_{\eps\in\{-1,1\}^{\N}}
 \frac1x\bigl|\{1\le n\le x:v_n(\eps)\ne v_{n+1}(\eps)\}\bigr|=0.
\]
In particular, for every $\eps\in\{-1,1\}^{\N}$, the set $\{n\in\N:v_n(\eps)=v_{n+1}(\eps)\}$ has asymptotic density one.
\end{theorem}

\section{Preliminaries}
We write $H_n=H_n(\eps)$ and $v_n=v_n(\eps)$, and put $H_0=0$ and $v_0=1$. The sequence $\eps$ is arbitrary. The sets and counting functions introduced below may depend on $\eps$, although this dependence is omitted from the notation. All indices in these sets and sums are integers.

Throughout the paper, $p$ denotes a prime. For a nonzero rational number $r$, write $r=p^b u/v$, where $b\in\Z$, $u,v\in\Z$ and $p\nmid uv$, and define $\vp(r)=b$. We set $\vp(0)=+\infty$. Put
\[
 \Zp=\{r\in\Q:\vp(r)\ge0\}
 =\{u/v:u,v\in\Z,\ p\nmid v\}.
\]
For $r,s\in\Zp$, the congruence $r\equiv s\pmod p$ means $r-s\in p\Zp$. Thus a rational number $u/v$ with $p\nmid v$ is reduced modulo $p$ by multiplying the residue of $u$ by the inverse of the residue of $v$ in $\Fp$. We only use rational congruences when the quantities being reduced belong to $\Zp$.

For products, we have $\vp(rs)=\vp(r)+\vp(s)$ for nonzero $r,s$, as follows directly from the definition.

For sums, we have $\vp(r+s)\ge\min\{\vp(r),\vp(s)\}$ for all rational $r,s$. When $\vp(r)\ne\vp(s)$, equality holds. To see this for nonzero $r,s$, put $a=\min\{\vp(r),\vp(s)\}$ and write $r=p^a u$, $s=p^a v$ with $u,v\in\Zp$. Since $u+v\in\Zp$, the inequality follows. If the two valuations are different, exactly one of $u,v$ belongs to $p\Zp$, so their sum is nonzero modulo $p$ and has valuation zero. This proves the equality assertion. If one summand is zero, both assertions follow from $r+0=r$.

If $p\mid v_n$, the reduced-fraction condition gives $p\nmid u_n(\eps)$, so $\vp(H_n)=-\vp(v_n)<0$. If $p\nmid v_n$, then $\vp(H_n)\ge0$. Consequently,
\begin{equation}\label{eq:denominator}
 \vp(v_n)=\max\{0,-\vp(H_n)\}\qquad(n\ge0).
\end{equation}
The convention at $n=0$ makes~\eqref{eq:denominator} valid there as well. All logarithms are natural. We write $F\ll G$ if $|F|\le CG$ for an absolute constant $C$, and $F\ll_L G$ if the constant may depend on $L$. Constants in these estimates never depend on $\eps$. We use Chebyshev's bound $\pi(t)\ll t/\log t$ for $t\ge2$, where $\pi(t)$ denotes the number of primes not exceeding $t$.

For integers $a\ge1$, $k\ge0$ and $M\ge0$, put
\[
 B_{p,a}(k)=p^aH_{kp^a},\qquad
 N_{p,a}(M)=|\{0\le k\le M:B_{p,a}(k)\in p\Zp\}|.
\]
The index $k=0$ is included in $N_{p,a}(M)$, since $B_{p,a}(0)=0$. We do not assume that every $B_{p,a}(k)$ belongs to $\Zp$; the condition $B_{p,a}(k)\in p\Zp$ is part of the definition of the counted set. For $x\ge1$, let $D_p(x)=|\{1\le m\le x:\vp(v_{m-1})\ne\vp(v_m)\}|$. The next lemma relates changes in the exponent of $p$ in a denominator to the indices counted by $N_{p,a}$.

\begin{lemma}\label{lem:reduction}
If $m$ is counted by $D_p(x)$, then $p\mid m$. Moreover, if $p^a\parallel m$, then
\begin{equation}\label{eq:event}
 \vp(H_{m-1})\ge1-a\quad\text{or}\quad\vp(H_m)\ge1-a.
\end{equation}
Consequently,
\begin{equation}\label{eq:reduction}
 D_p(x)\le2\sum_{\substack{a\ge1\\p^a\le x}}
 N_{p,a}\left(\left\lfloor\frac{x}{p^a}\right\rfloor\right).
\end{equation}
\end{lemma}
\begin{proof}
We have $H_m=H_{m-1}+\eps_m/m$. Suppose first that $p\nmid m$. Then $\vp(\eps_m/m)=0$. If $\vp(H_{m-1})<0$, the two summands have different valuations, so $\vp(H_m)=\vp(H_{m-1})$. If $\vp(H_{m-1})\ge0$, both summands belong to $\Zp$, and hence $\vp(H_m)\ge0$. In the first case, the two negative valuations are equal; in the second, neither reduced denominator is divisible by $p$. Thus~\eqref{eq:denominator} gives $\vp(v_m)=\vp(v_{m-1})$ in both cases. Every index counted by $D_p(x)$ must therefore be divisible by $p$.

Now let $m$ be such an index and write $p^a\parallel m$ with $a\ge1$. Put $s=\vp(H_{m-1})$ and $t=\vp(H_m)$. Since $\vp(\eps_m/m)=-a$, there are three possibilities. If $s<-a$, the unequal-valuation rule gives $t=s$, and the denominator exponents are equal, contrary to the choice of $m$. If $s=-a$, then $t\ge-a$; equality would again give equal denominator exponents, so $t>-a$. If $s>-a$, this inequality already holds for the earlier sum, and the unequal-valuation rule gives $t=-a$. Thus $s>-a$ or $t>-a$. Valuations lie in $\Z\cup\{+\infty\}$, so these inequalities imply~\eqref{eq:event}.

Write $m=kp^a$. No integer strictly between $(k-1)p^a$ and $kp^a$ is divisible by $p^a$. Therefore each such $i$ satisfies $\vp(i)\le a-1$, and each summand below has valuation $a-\vp(i)\ge1$. It follows that
\[
 p^a(H_{m-1}-H_{(k-1)p^a})
 =\sum_{(k-1)p^a<i<kp^a}\frac{\eps_i p^a}{i}\in p\Zp.
\]
If the first alternative in~\eqref{eq:event} holds, then $p^aH_{m-1}\in p\Zp$. Subtracting the preceding difference gives $B_{p,a}(k-1)\in p\Zp$. If the second alternative holds, then $B_{p,a}(k)=p^aH_m\in p\Zp$.

For fixed $a$, put $M=\lfloor x/p^a\rfloor$. Each possible $m$ corresponds to one $k$ with $1\le k\le M$ and $p\nmid k$. Omitting the latter restriction can only increase the count. The map $k\mapsto k-1$ sends the indices satisfying the first alternative injectively into the set counted by $N_{p,a}(M)$, since $0\le k-1\le M-1$. The indices satisfying the second alternative form a subset of that counted set. Their union therefore has at most $2N_{p,a}(M)$ elements, even if an index satisfies both alternatives. Summing over the possible exponents $a$ proves~\eqref{eq:reduction}.
\end{proof}

The following lemma follows from the interval argument in the proof of Wu and Yan~\cite[Lemma~2.3]{WY}. We give the argument for arbitrary signs.

\begin{lemma}\label{lem:fixed}
Let $p$ be a prime and $k\ge0$ an integer. Put $E_{p,k}=\{n\ge0:\vp(H_n)\ge-k\}$ and $c_p=\lceil p/2\rceil$. For every integer $B\ge k+1$, we have
\[
 |E_{p,k}\cap[0,p^B)|\le p^{k+1}c_p^{B-k-1}.
\]
In particular, $|E_{p,k}\cap[0,x]|=o(x)$ uniformly in $\eps$ for each fixed $p$ and $k$.
\end{lemma}
\begin{proof}
We first bound how many of the $p$ equal subintervals of a given interval can meet $E_{p,k}$. Let $b\ge k+1$ and $j\ge0$ be integers, and put $t_r=(jp+r)p^b$ and $I_r=[t_r,t_r+p^b)$ for $0\le r<p$. These half-open intervals partition $[jp^{b+1},(j+1)p^{b+1})$.

If $n\in I_r$ is an integer, there is no multiple of $p^b$ in $(t_r,n]$. Thus every term in
\[
 p^b(H_n-H_{t_r})=\sum_{t_r<i\le n}\frac{\eps_i p^b}{i}
\]
has valuation at least $1$, and the entire sum belongs to $p\Zp$. If also $n\in E_{p,k}$, then $\vp(p^bH_n)=b+\vp(H_n)\ge b-k\ge1$. Subtracting the displayed difference from $p^bH_n$ shows that $p^bH_{t_r}\in p\Zp$. We have therefore proved the implication
\[
 I_r\cap E_{p,k}\ne\varnothing\quad\Longrightarrow\quad p^bH_{t_r}\in p\Zp.
\]

For $1\le r<p$, the only multiple of $p^b$ in $(t_{r-1},t_r]$ is $t_r$. Its coefficient after multiplication by $p^b$ is $\eps_{t_r}/(jp+r)$, and $p\nmid jp+r$. All the other summands belong to $p\Zp$. Consequently,
\[
 p^b(H_{t_r}-H_{t_{r-1}})
 \equiv\frac{\eps_{t_r}}{jp+r}\not\equiv0\pmod p.
\]
If both $I_{r-1}$ and $I_r$ met $E_{p,k}$, the implication above would put both $p^bH_{t_{r-1}}$ and $p^bH_{t_r}$ in $p\Zp$. Their difference would then be zero modulo $p$, contradicting the last congruence. Thus no two consecutive intervals can meet $E_{p,k}$. Pair $I_0$ with $I_1$, $I_2$ with $I_3$, and so on. Each pair contributes at most one interval, with at most one unpaired interval left. At most $\lceil p/2\rceil=c_p$ intervals can therefore meet $E_{p,k}$.

If $B=k+1$, the original interval contains $p^{k+1}$ integers, giving the claimed bound. Suppose that $B>k+1$. Starting with $[0,p^B)$, divide every interval retained at one stage into $p$ equal subintervals, retaining only those meeting $E_{p,k}$. Every parent interval has the form $[jp^{b+1},(j+1)p^{b+1})$, so the preceding argument applies at each stage. The exponents used are $b=B-1,B-2,\ldots,k+1$, giving $B-k-1$ divisions. Each retained parent contributes at most $c_p$ children. At the last stage, at most $c_p^{B-k-1}$ intervals remain, each of length $p^{k+1}$ and containing exactly $p^{k+1}$ integers. Every integer of $E_{p,k}\cap[0,p^B)$ remains in one of these intervals, which proves the bound.

For sufficiently large $x$, choose $B$ with $p^{B-1}\le x<p^B$. Then
\[
 \sup_\eps\frac{|E_{p,k}\cap[0,x]|}{x}
 \le\frac{p^{k+1}c_p^{B-k-1}}{p^{B-1}}
 =p\left(\frac{c_p}{p}\right)^{B-k-1}\longrightarrow0.
\]
Here $B\to\infty$ as $x\to\infty$, while $p$ and $k$ remain fixed. Since $c_p<p$, the final expression tends to zero. It does not depend on $\eps$, which also proves the asserted uniformity.
\end{proof}

For $d\ge1$ and $\sigma=(\sigma_1,\ldots,\sigma_d)\in\{-1,1\}^d$, let
\[
 P_{d,\sigma}(T)=\sum_{s=1}^d\sigma_s
 \prod_{\substack{1\le i\le d\\i\ne s}}(T+i).
\]
We first record a consequence of K\"ursch\'ak's $2$-adic argument~\cite{K}. For signed sums starting at $1$, see also Bettin, Molteni and Sanna~\cite[Proposition~2.6]{BMS}.

\begin{lemma}\label{lem:nonzero}
For every integer $d\ge1$, every $\sigma\in\{-1,1\}^d$ and every integer $k\ge0$, we have $0<|P_{d,\sigma}(k)|\le d(k+d)^{d-1}$.
\end{lemma}
\begin{proof}
We first prove that exactly one of $k+1,\ldots,k+d$ has largest $2$-adic valuation. A largest valuation exists because this is a finite nonempty set of positive integers. Suppose that two of its members $u<v$ have this valuation $b$. Write $u=2^b(2r+1)$ and $v=2^b(2s+1)$ with integers $s\ge r+1$. Then $v-u\ge2^{b+1}$, and hence
\[
 u<u+2^b=2^{b+1}(r+1)<v.
\]
The integer $u+2^b$ lies in the same consecutive interval and has valuation at least $b+1$, contradicting the maximality of $b$.

Write $k+j=2^{b_j}u_j$ with $u_j$ odd, and let $j_0$ be the unique index for which $b=b_{j_0}=\max_j b_j$. After multiplying the signed reciprocal sum by $2^b$, we have
\[
 2^b\sum_{j=1}^d\frac{\sigma_j}{k+j}
 =\frac{\sigma_{j_0}}{u_{j_0}}
 +\sum_{j\ne j_0}\frac{\sigma_j2^{b-b_j}}{u_j}.
\]
For $j\ne j_0$, the integer $b-b_j$ is at least $1$, so the corresponding summand belongs to $2\Z_{(2)}$. The remaining summand has odd numerator and denominator, and is therefore congruent to $1$ modulo $2$. Thus
\[
 2^b\sum_{j=1}^d\frac{\sigma_j}{k+j}
 \equiv1\pmod2.
\]
In particular, the signed sum is nonzero. Since every $k+i$ is positive, multiplying by their product preserves nonvanishing and gives
\[
 P_{d,\sigma}(k)=\left(\prod_{i=1}^d(k+i)\right)
 \left(\sum_{j=1}^d\frac{\sigma_j}{k+j}\right)\ne0.
\]
The definition also shows that $P_{d,\sigma}(k)$ is an integer. Finally, each product in that definition has $d-1$ factors, all at most $k+d$. By the triangle inequality,
\[
 |P_{d,\sigma}(k)|\le\sum_{j=1}^d\prod_{\substack{1\le i\le d\\i\ne j}}(k+i)
 \le d(k+d)^{d-1}.
\]
When $d=1$, the empty product equals $1$ and $P_{1,\sigma}(k)=\sigma_1$, so the same argument and bound remain valid.
\end{proof}

The next lemma uses the short-gap argument of Sanna~\cite[Lemma~2.2]{Sanna} and Wu and Chen~\cite[Lemma~2.2]{WC17}. For signed versions, see Wu and Chen~\cite[Lemma~3.2]{WC} and Wu and Yan~\cite[Lemmas~2.1--2.2]{WY}. Here we include all sign patterns.

\begin{lemma}\label{lem:returns}
For integers $a\ge1$, $M\ge0$ and $1\le L<p$, we have
\begin{equation}\label{eq:returns}
 N_{p,a}(M)\le\frac{M}{L}
 +\left(\frac{M}{p}+1\right)(1+R(L)),
\end{equation}
where $R(L)=\sum_{d=1}^L2^d(d-1)\le2^{L+1}L$.
\end{lemma}
\begin{proof}
For each integer $j\ge0$, list the indices counted by $N_{p,a}(M)$ in $[jp,(j+1)p)$ as $k_{j,1}<\cdots<k_{j,s_j}$. Empty lists contribute nothing. A nonempty list has one first index, and each of its other indices is the right endpoint of exactly one gap $d_{j,r}=k_{j,r+1}-k_{j,r}$, where $1\le r<s_j$. Thus it suffices to count first indices, gaps longer than $L$, and gaps of length at most $L$.

Since every counted index lies in $[0,M]$, only the blocks with $0\le j\le\lfloor M/p\rfloor$ can be nonempty. The number of first indices is at most $\lfloor M/p\rfloor+1\le M/p+1$. Within one block, the sum of its gaps telescopes to $k_{j,s_j}-k_{j,1}$. These intervals have disjoint interiors as $j$ varies and all lie in $[0,M]$. Therefore,
\[
 \sum_j\sum_{r=1}^{s_j-1}d_{j,r}
 =\sum_{j:s_j\ge2}(k_{j,s_j}-k_{j,1})\le M.
\]
Every gap longer than $L$ contributes more than $L$ to this sum, so there are at most $M/L$ such gaps.

Now suppose that $k$ and $k+d$ are consecutive indices in the same block and $1\le d\le L$. Because $jp\le k<k+d<(j+1)p$, none of $k+1,\ldots,k+d$ is divisible by $p$. Among the indices $i$ in $(kp^a,(k+d)p^a]$, the multiples of $p^a$ are precisely $(k+s)p^a$ with $1\le s\le d$. Separating these terms gives
\[
 B_{p,a}(k+d)-B_{p,a}(k)
 =\sum_{s=1}^d\frac{\eps_{(k+s)p^a}}{k+s}
 +\sum_{\substack{kp^a<i\le(k+d)p^a\\p^a\nmid i}}\frac{\eps_i p^a}{i}.
\]
The first sum lies in $\Zp$, since its denominators are prime to $p$. Every term in the second sum has valuation $a-\vp(i)\ge1$. Both $B_{p,a}(k)$ and $B_{p,a}(k+d)$ belong to $p\Zp$ by the definition of the lists. Reducing the displayed identity modulo $p$ therefore gives
\[
 0\equiv B_{p,a}(k+d)-B_{p,a}(k)
 \equiv\sum_{s=1}^d\frac{\eps_{(k+s)p^a}}{k+s}\pmod p.
\]
Put $\sigma_s=\eps_{(k+s)p^a}$. Multiplying the congruence by $\prod_{i=1}^d(k+i)$, which is not divisible by $p$, gives
\begin{equation}\label{eq:root}
 P_{d,\sigma}(k)\equiv0\pmod p.
\end{equation}
Each summand of $P_{d,\sigma}$ has degree $d-1$, so its degree is at most $d-1$. At $T=-1$, all summands except the one with $s=1$ contain the zero factor $T+1$. Hence $P_{d,\sigma}(-1)=\sigma_1(d-1)!$. Since $d\le L<p$, none of $1,\ldots,d-1$ is divisible by $p$, and this value is nonzero modulo $p$. Thus the reduction of $P_{d,\sigma}$ is a nonzero polynomial over $\Fp$. By the factor theorem, a nonzero polynomial of degree at most $d-1$ has at most $d-1$ roots. This includes $d=1$, when the polynomial is the nonzero constant $\sigma_1$ and has no roots.

For each fixed $d$, there are $2^d$ possible sign patterns. The union of the root classes for all $1\le d\le L$ and all these patterns therefore contains at most $\sum_{d=1}^L2^d(d-1)=R(L)$ residue classes modulo $p$. Overlap between root sets can only decrease this number. By~\eqref{eq:root}, every starting index of a short gap belongs to one of these classes.

An individual residue class contains at most $\lfloor M/p\rfloor+1$ integers in $[0,M]$, because consecutive representatives are $p$ apart. Moreover, a starting index belongs to only one block and has at most one successor in its ordered list. Thus the number of short gaps is at most the number of possible starting indices, which is at most $(M/p+1)R(L)$. Adding the bounds for first indices, long gaps and short gaps proves~\eqref{eq:returns}. Finally,
\[
 R(L)\le L\sum_{d=1}^L2^d=L(2^{L+1}-2)\le2^{L+1}L.
\]
All sign patterns were included, so the bound is independent of the original sequence $\eps$.
\end{proof}

\section{Proofs}
\begin{proof}[Proof of Theorem~\ref{thm:main}]
Put $D_\eps(x)=|\{1\le m\le x:v_{m-1}\ne v_m\}|$. We first reduce the desired estimate to a sum over primes. For $q_m=\lcm(1,\ldots,m)$, the identity
\[
 H_m=\frac1{q_m}\sum_{i=1}^m\eps_i\frac{q_m}{i}
\]
has an integer numerator, so $v_m\mid q_m$. The same argument gives $v_{m-1}\mid q_m$, with $v_0=1$ when $m=1$. If the two positive denominators are different, their prime factorizations differ at some prime $p$ dividing $q_m$, and necessarily $p\le m$. Every index counted by $D_\eps(x)$ is therefore counted by at least one $D_p(x)$ with $p\le x$. It follows that
\begin{equation}\label{eq:union}
 D_\eps(x)\le\sum_{p\le x}D_p(x).
\end{equation}
An index may be counted for several primes, which is why~\eqref{eq:union} is an inequality. It suffices to prove that its right-hand side is $o(x)$ uniformly in $\eps$.

Fix integers $Y\ge256$ and $L\ge1$, and take $x$ sufficiently large that $\sqrt x>\max\{Y,L,2\}$. These two parameters will remain fixed while $x$ tends to infinity. For the primes up to $Y$, we use the fixed-prime estimate of Lemma~\ref{lem:fixed}. For the intermediate primes, we apply Lemma~\ref{lem:returns} with a length depending on $p$. For primes exceeding $\sqrt x$, we keep the length $L$ fixed and count prime divisors of the nonzero integers in Lemma~\ref{lem:nonzero}.

\medskip\noindent\textit{Case 1. $p\le Y$.}
Fix a prime $p\le Y$ and an integer $A\ge1$. Let $m$ be counted by $D_p(x)$ and put $a=\vp(m)$. Lemma~\ref{lem:reduction} gives $a\ge1$. If $a\le A$, then $1-a\ge1-A$, so~\eqref{eq:event} gives $\vp(H_{m-1})\ge1-A$ or $\vp(H_m)\ge1-A$. By the definition of $E_{p,A-1}$, this means $m-1\in E_{p,A-1}$ or $m\in E_{p,A-1}$. Both indices lie in $[0,x]$. Each index in this latter set can correspond to at most one $m$ by each of the maps $m\mapsto m-1$ and $m\mapsto m$. This accounts for at most $2|E_{p,A-1}\cap[0,x]|$ values of $m$.

If $a\ge A+1$, then $p^{A+1}\mid m$. The number of such positive integers $m\le x$ is at most $\lfloor x/p^{A+1}\rfloor$. Combining the two possibilities, we obtain
\[
 D_p(x)\le2|E_{p,A-1}\cap[0,x]|+\left\lfloor\frac{x}{p^{A+1}}\right\rfloor.
\]
Divide by $x$ and take the supremum over $\eps$. For fixed $p$ and $A$, the first term tends to zero by Lemma~\ref{lem:fixed}. Hence
\[
 \limsup_{x\to\infty}\sup_\eps\frac{D_p(x)}x\le p^{-A-1}.
\]
This holds for every integer $A\ge1$. Letting $A\to\infty$ shows that the nonnegative quantity $\sup_\eps D_p(x)/x$ tends to zero for each fixed $p$. There are only finitely many primes up to the fixed number $Y$, and therefore
\begin{equation}\label{eq:fixedprimes}
 e_Y(x):=\sup_\eps\frac1x\sum_{p\le Y}D_p(x)
 \le\sum_{p\le Y}\sup_\eps\frac{D_p(x)}x\longrightarrow0.
\end{equation}

\medskip\noindent\textit{Case 2. $Y<p\le\sqrt x$.}
Put $L_p=\lfloor(\log p)/(4\log2)\rfloor$ and $A_p=\lfloor(\log x)/(\log p)\rfloor$. The exponents occurring in~\eqref{eq:reduction} are precisely $1\le a\le A_p$. We first check the size of $L_p$. Since $p>256$, the number $z=(\log p)/(4\log2)$ is greater than $2$. For $z\ge2$, we have $\lfloor z\rfloor\ge z-1\ge z/2$. Thus
\[
 \frac{\log p}{8\log2}\le L_p\le\frac{\log p}{4\log2}<p.
\]
In particular, $1\le L_p<p$, as required in Lemma~\ref{lem:returns}. Moreover, $2^{L_p}\le2^{(\log p)/(4\log2)}=p^{1/4}$, and therefore
\[
 1+R(L_p)\le1+2^{L_p+1}L_p\ll p^{1/4}\log p.
\]
The lower bound for $L_p$ controls the long gaps in Lemma~\ref{lem:returns}, while this upper bound controls its short-gap contribution.

Apply that lemma with $M=\lfloor x/p^a\rfloor$ and $L=L_p$. Since $M\le x/p^a$, it gives
\[
 N_{p,a}\left(\left\lfloor\frac{x}{p^a}\right\rfloor\right)
 \le\frac{x}{p^aL_p}+\left(\frac{x}{p^{a+1}}+1\right)(1+R(L_p)).
\]
Substituting the two bounds for $L_p$ and $R(L_p)$, we obtain, with an absolute implied constant,
\[
 N_{p,a}\left(\left\lfloor\frac{x}{p^a}\right\rfloor\right)
 \ll\frac{x}{p^a\log p}+\frac{x\log p}{p^{a+3/4}}+p^{1/4}\log p.
\]
To sum these terms over $a$, we first use
\[
 \sum_{a=1}^{A_p}p^{-a}\le\sum_{a=1}^\infty p^{-a}
 =\frac{p^{-1}}{1-p^{-1}}=\frac1{p-1}.
\]
The first two contributions are therefore bounded by
\[
 \sum_{a=1}^{A_p}\frac{x}{p^a\log p}\le\frac{x}{(p-1)\log p},
 \qquad
 \sum_{a=1}^{A_p}\frac{x\log p}{p^{a+3/4}}\le\frac{x\log p}{p^{3/4}(p-1)}.
\]
The third term does not depend on $a$, and $A_p\log p\le\log x$. Hence $\sum_{a=1}^{A_p}p^{1/4}\log p\le p^{1/4}\log x$. Applying Lemma~\ref{lem:reduction} to the summed estimate, and absorbing its factor $2$ into the absolute implied constant, we obtain
\[
 D_p(x)\ll\frac{x}{(p-1)\log p}
 +\frac{x\log p}{p^{3/4}(p-1)}+p^{1/4}\log x.
\]
The ratio of the second term to the first is $(\log p)^2/p^{3/4}$. This is bounded by an absolute constant: on putting $u=\log p$, the function becomes $u^2e^{-3u/4}$, which has its maximum for $u\ge0$ at $u=8/3$. Indeed, its derivative is $ue^{-3u/4}(2-3u/4)$. We may therefore absorb the second term into the first. Since $p-1\ge p/2$, we have
\[
 D_p(x)\ll\frac{x}{p\log p}+p^{1/4}\log x.
\]

We next estimate the sum of $1/(p\log p)$. For $Z>Y$, let $f(t)=1/(t\log t)$. Partial summation gives
\[
 \sum_{Y<p\le Z}f(p)=\pi(Z)f(Z)-\pi(Y)f(Y)-\int_Y^Z\pi(t)f'(t)\,dt.
\]
Here $-f'(t)=(\log t+1)/(t^2(\log t)^2)$. Chebyshev's bound gives $\pi(Z)f(Z)\ll1/(\log Z)^2$. Also, since $t\ge Y\ge256$, we have $\log t+1\le2\log t$, and hence
\[
 -\pi(t)f'(t)\ll\frac{\log t+1}{t(\log t)^3}
 \le\frac{2}{t(\log t)^2}.
\]
The lower boundary term $-\pi(Y)f(Y)$ is nonpositive and may be omitted for an upper bound. Thus the nonnegative partial sums over $Y<p\le Z$ are bounded above by a convergent improper integral and a term tending to zero. Letting $Z\to\infty$ proves both convergence and the bound
\[
 \sum_{p>Y}\frac1{p\log p}
 \ll\int_Y^\infty\frac{dt}{t(\log t)^2}
 =\frac1{\log Y}.
\]
For the remaining term, there are at most $\sqrt x$ primes up to $\sqrt x$, and each satisfies $p^{1/4}\le x^{1/8}$. Consequently, $\sum_{p\le\sqrt x}p^{1/4}\le x^{5/8}$. Summing the bound for $D_p(x)$ over $Y<p\le\sqrt x$, we obtain
\begin{equation}\label{eq:middle}
 \sum_{Y<p\le\sqrt x}D_p(x)
 \ll\frac{x}{\log Y}+x^{5/8}\log x.
\end{equation}

\medskip\noindent\textit{Case 3. $\sqrt x<p\le x$.}
Put $M_p=\lfloor x/p\rfloor$. Since $p^2>x$, only $a=1$ occurs in~\eqref{eq:reduction}, and hence $D_p(x)\le2N_{p,1}(M_p)$. Moreover, $M_p\le x/p<\sqrt x<p$. All the indices counted by $N_{p,1}(M_p)$ therefore lie in the single block $[0,p)$.

List these indices in increasing order. As in Lemma~\ref{lem:returns}, each is either the first index or the right endpoint of one gap. There is at most one first index, and the sum of the gap lengths is at most $M_p$. Thus at most $M_p/L$ gaps have length greater than $L$. Let $S_L(x)$ be the total number of gaps of length at most $L$, summed over all primes $\sqrt x<p\le x$. The first indices contribute at most $\pi(x)$, and the long gaps contribute at most $\sum_{\sqrt x<p\le x}M_p/L$. Using $M_p\le x/p$, we obtain
\begin{equation}\label{eq:largecount}
 \sum_{\sqrt x<p\le x}N_{p,1}(M_p)
 \le\pi(x)+\frac{x}{L}\sum_{\sqrt x<p\le x}\frac1p+S_L(x).
\end{equation}
To estimate the prime reciprocal sum here, partial summation gives the exact identity
\[
 \sum_{\sqrt x<p\le x}\frac1p
 =\frac{\pi(x)}x-\frac{\pi(\sqrt x)}{\sqrt x}
 +\int_{\sqrt x}^x\frac{\pi(t)}{t^2}\,dt.
\]
Omit the nonpositive lower boundary term and apply Chebyshev's bound to the other two terms. Then
\[
 \sum_{\sqrt x<p\le x}\frac1p
 \ll\frac1{\log x}+\int_{\sqrt x}^x\frac{dt}{t\log t}.
\]
The integral equals $\log\log x-\log\log\sqrt x=\log2$, because $\log\sqrt x=(\log x)/2$. The prime reciprocal sum is therefore bounded by an absolute constant. In particular, this constant does not depend on $L$.

It remains to show that $S_L(x)\ll_L\sqrt x$. A gap from $k$ to $k+d$ counted by $S_L(x)$ satisfies $0\le k<k+d\le M_p<\sqrt x$ and $1\le d\le L<p$. Both endpoints lie in $[0,p)$, so the argument leading to~\eqref{eq:root}, now with $a=1$, applies. Taking $\sigma_s=\eps_{(k+s)p}$, we have $p\mid P_{d,\sigma}(k)$.

The sign pattern associated with a gap may vary with $p$. For each $d$, however, there are only $2^d$ possible patterns, and we count all of them. We reverse the order of counting: fix $k,d,\sigma$ first, and then count the primes $p>\sqrt x$ which can divide $P_{d,\sigma}(k)$. For a fixed triple, this integer no longer depends on $p$ or on the original sequence of signs.

Put $Q=|P_{d,\sigma}(k)|$. By Lemma~\ref{lem:nonzero}, it is a positive integer. Since $d\le L$ and $k\le\sqrt x$, the same lemma gives
\[
 1\le Q\le d(k+d)^{d-1}\le L(\sqrt x+L)^{L-1}.
\]
Our choice of $x$ ensures $L\le\sqrt x$ and $2\sqrt x\le x$. Therefore,
\[
 Q\le L(2\sqrt x)^{L-1}\le\sqrt x\,x^{L-1}\le x^L.
\]
If $Q=1$, it has no prime divisors. Now suppose $Q>1$. Let $p_1,\ldots,p_r$ be its distinct prime divisors greater than $\sqrt x$. If $r=0$, there is nothing to count. If $r\ge1$, their product divides $Q$, and consequently
\[
 x^{r/2}<p_1\cdots p_r\le Q\le x^L.
\]
As $x>1$, taking logarithms gives $r<2L$. Thus, for every fixed triple $k,d,\sigma$, at most $2L$ primes can contribute a gap with those data.

There are at most $\lfloor\sqrt x\rfloor+1$ possible integers $k$. For each $d$ with $1\le d\le L$, there are $2^d$ possible sign patterns. For a fixed prime $p$, each starting index $k$ belongs to only one ordered list and has only one possible next counted index. Thus two distinct gaps cannot give the same tuple $(p,k,d,\sigma)$. Counting all such tuples, including triples that do not actually arise from a gap, gives the upper bound
\[
 S_L(x)\le2L(\lfloor\sqrt x\rfloor+1)\sum_{d=1}^L2^d
 \ll_L\sqrt x.
\]
The implied constant depends only on $L$, since the finite sum of sign-pattern counts is $\sum_{d=1}^L2^d=2^{L+1}-2$. In particular, it is independent of $\eps$. Inserting this estimate and the prime reciprocal bound into~\eqref{eq:largecount}, and using $\pi(x)\ll x/\log x$ and $D_p(x)\le2N_{p,1}(M_p)$, we obtain
\begin{equation}\label{eq:large}
 \sum_{\sqrt x<p\le x}D_p(x)
 \le C\left(\frac{x}{\log x}+\frac{x}{L}\right)+C_L\sqrt x,
\end{equation}
where $C$ is absolute and $C_L$ depends only on $L$.

Combining~\eqref{eq:union}, \eqref{eq:fixedprimes}, \eqref{eq:middle} and~\eqref{eq:large}, we have
\[
 \sup_\eps\frac{D_\eps(x)}x
 \le e_Y(x)+C\left(\frac1{\log Y}+x^{-3/8}\log x
 +\frac1{\log x}+\frac1L\right)+C_Lx^{-1/2}.
\]
For fixed $Y$ and $L$, the term $e_Y(x)$ tends to zero by~\eqref{eq:fixedprimes}. The terms $x^{-3/8}\log x$, $1/\log x$ and $C_Lx^{-1/2}$ also tend to zero. All these bounds hold uniformly in $\eps$. Taking the upper limit in $x$, we obtain
\[
 \limsup_{x\to\infty}\sup_\eps\frac{D_\eps(x)}x
 \le C\left(\frac1{\log Y}+\frac1L\right).
\]
The constant $C$ is independent of $Y$ and $L$. Given $\delta>0$, choose fixed integers $Y\ge256$ and $L\ge1$ so large that $C(1/\log Y+1/L)<\delta/2$. With these choices held fixed, take $x$ sufficiently large that
\[
 e_Y(x)+Cx^{-3/8}\log x+\frac C{\log x}+C_Lx^{-1/2}<\frac\delta2.
\]
This threshold for $x$ does not depend on $\eps$. Substitution into the displayed bound for $\sup_\eps D_\eps(x)/x$ now gives $\sup_\eps D_\eps(x)/x<\delta$. Since $\delta$ is arbitrary, $\sup_\eps D_\eps(x)/x\to0$.

Finally, the change of index $m=n+1$ takes the integers $1\le n\le x$ to the integers $2\le m\le x+1$. The omitted index $m=1$ is not counted by $D_\eps$, because $v_0=v_1=1$. Hence, for every sequence $\eps$,
\[
 |\{1\le n\le x:v_n\ne v_{n+1}\}|=D_\eps(x+1).
\]
After taking the supremum and dividing by $x$, this gives
\[
 \sup_\eps\frac{|\{1\le n\le x:v_n\ne v_{n+1}\}|}{x}
 =\frac{x+1}{x}\sup_\eps\frac{D_\eps(x+1)}{x+1}\longrightarrow0.
\]
For any fixed sequence, exactly $D_\eps(x+1)$ of the $\lfloor x\rfloor$ indices have unequal consecutive denominators. The proportion with equal denominators is therefore $\lfloor x\rfloor/x-D_\eps(x+1)/x$, which tends to $1$. This  completes the proof of Theorem~\ref{thm:main}.
\end{proof}

\end{document}